\documentclass[12pt]{amsart}

\usepackage[utf8]{inputenc}
\usepackage[OT2,T1]{fontenc}
\usepackage[english]{babel}
\usepackage{lmodern}

\usepackage{amsmath}
\usepackage{amsfonts}
\usepackage{amssymb}
\usepackage{amsthm}
\usepackage{mathrsfs}

\usepackage[OT2,T1]{fontenc}
\DeclareSymbolFont{cyrletters}{OT2}{wncyr}{m}{n}
\DeclareMathSymbol{\Sha}{\mathalpha}{cyrletters}{"58}

\usepackage{url}
\usepackage[colorlinks=true]{hyperref}

\usepackage{multirow}

\usepackage{graphicx}
\usepackage{tikz-cd}

\usepackage{etoolbox}

\numberwithin{equation}{subsection}

\theoremstyle{plain}

\newtheorem{proposition}[equation]{Proposition}
\newtheorem{lemma}[equation]{Lemma}
\newtheorem{corollary}[equation]{Corollary}

\theoremstyle{definition}

\theoremstyle{remark}
\newtheorem{remark}[equation]{Remark}

\def\defi{\textsf}
\def\ext{\!\mid\!}

\def\epsilon{\varepsilon}
\def\theta{\vartheta}
\def\phi{\varphi}

\def\Belyi{Bely\u{\i}}

\def\Magma{\textsc{Magma}}

\DeclareMathOperator{\im}{im}

\DeclareMathOperator{\PGamma}{P \Gamma}
\DeclareMathOperator{\PDelta}{P \Delta}

\DeclareMathOperator{\PSL}{PSL}

\DeclareMathOperator{\SL}{SL}

\DeclareMathOperator{\Stab}{Stab}

\DeclareMathOperator{\Sym}{Sym}

\DeclareMathOperator{\Tr}{Tr}

\def\C{\mathbb{C}}

\def\F{\mathbb{F}}

\def\H{\mathcal{H}}

\def\O{\mathcal{O}}
\def\P{\mathbb{P}}
\def\Q{\mathbb{Q}}
\def\R{\mathbb{R}}

\def\Z{\mathbb{Z}}

\def\fbar{\overline{f}}

\def\xbar{\overline{x}}

\def\xbar{\overline{x}}
\def\Xbar{\overline{X}}
\def\ybar{\overline{y}}
\def\Ybar{\overline{Y}}
\def\Zbar{\overline{Z}}
\def\Gammabar{\overline{\Gamma}}
\def\lambdabar{\overline{\lambda}}

\usepackage{color}

\newcommand{\smat}[4]{
  \left(
    \begin{smallmatrix}
      #1 & #2\\
      #3 & #4
    \end{smallmatrix}
  \right)
}

\hypersetup{
  pdfauthor   = {Jeroen Sijsling},
  pdftitle    = {Canonical models of arithmetic (1; inf)-curves},
  pdfsubject  = {},
  pdfkeywords = {},
  backref=true, pagebackref=true, hyperindex=true, colorlinks=true,
  breaklinks=true, urlcolor=blue, linkcolor=blue, citecolor=blue,
  bookmarks=true, bookmarksopen=true}

\begin{document}

\title{Canonical models of arithmetic $(1; \infty)$-curves}

\author{Jeroen Sijsling}
\address{Institut für Reine Mathematik, Universität Ulm, Helmholtzstrasse 18,
89081 Ulm, Germany}
\email{jeroen.sijsling@uni-ulm.de}

\date{\today}

\begin{abstract}
  In $1983$ Takeuchi showed that up to conjugation there are exactly $4$
  arithmetic subgroups of $\PSL_2 (\R)$ with signature $(1; \infty)$. Shinichi
  Mochizuki gave a purely geometric characterization of the corresponding
  arithmetic $(1; \infty)$-curves, which also arise naturally in the context of
  his recent work on inter-universal Teichmüller theory.
  
  Using \Belyi\ maps, we explicitly determine the canonical models of these
  curves. We also study their arithmetic properties and modular
  interpretations.
\end{abstract}


\maketitle

{\small\tableofcontents}

\setcounter{subsection}{1}

Let $(E, O)$ be a pointed curve of genus $1$ over $\C$, and let $e \in \Z_{\ge
2} \cup \left\{ \infty \right\}$. We can then construct the universal ramified
cover $U \to E$ that ramifies over $O$ with index $e$. As a complex analytic
space $U$ is isomorphic to the upper half plane $\H$, and the covering map $\H
\to E$ can be described as a quotient by a subgroup $\PGamma$ of $\PSL_2 (\R)$.
As a subgroup of $\PSL_2 (\R)$, the group $\PGamma$ is well-defined up to
conjugacy. We can ask when the pair $(E, O)$ and the ramification index $e$
determine an \defi{arithmetic} group $\PGamma$. This is the same as asking
which arithmetic subgroups $\PGamma$ of $\PSL_2 (\R)$ have signature $(1; e)$,
or alternatively a presentation of the form
\begin{equation}
  \PGamma = \langle \alpha, \beta \; \ext \; [ \alpha, \beta ]^e = 1 \rangle .
\end{equation}

This question was answered by Takeuchi \cite{takeuchi-1e}, who provided the
finite list of arithmetic groups of signature $(1; e)$. When $e$ is finite,
then the group $\PGamma$ has no cusps and comes from a non-trivial quaternion
algebra. The resulting quotient curves
\begin{equation}
  X (\PGamma) = \PGamma \backslash \H
\end{equation}
were studied in \cite{sijsling-1e}, where their canonical models in the sense
of Shimura \cite{shimura} were determined.

The present article considers the curves in Takeuchi's list that were not
treated in \cite{sijsling-1e}, namely those for which $e = \infty$. This is the
case for $4$ conjugacy classes of groups $\PGamma$. These groups have a single
conjugacy class of cusps, and are commensurable with the modular group $\PGamma
(1) = \PSL_2 (\Z)$. This means that classical methods are available to compute
canonical models of the corresponding quotient curves $X (\PGamma)$.

It should be mentioned that determining these $4$ curves did not require any
fundamentally new theory or concepts to be developed. Still, our computations
were rather involved in practice, and there seems to be no way to determine
equations for these curves by a short theoretical detour. We have used the
computer algebra system \Magma\ \cite{magma} to facilitate these calculations.
All the code needed to obtain the result in this article can be found online at
\cite{sijsling-1inf-code}. Along the way, we find some interesting phenomena:
an explicit Shimura curve whose Atkin-Lehner quotients has smaller canonical
field of moduli than the original curve (case III below), and a modular
description of the cover corresponding to the commutator subgroup of $\SL_2
(\Z)$ (case IV below).

However, the main motivation for this computation is not simply that these $4$
curves ``are there'', but that they play a role in the work of Shinichi
Mochizuki. In the articles \cite{mochizuki-corresps, mochizuki-canonical} he
showed that the property of these curves being arithmetic can be rephrased
purely geometrically, in terms of a certain category $\overline{\text{Loc}}
(X)$ introduced in \cite[Definition 2.1]{mochizuki-canonical} not having a
terminal object (or \defi{core}). In Mochizuki's later inter-universal
Teichmüller theory these curves are exceptional, in the sense that the theory
developed does not apply to these curves and the corresponding categories,
requiring them to be excluded from consideration. For instances of this
happening in the literature on this theory, we refer to Fesenko's expository
work \cite[Footnote 27]{fesenko} as well as the original work of Mochizuki
\cite[Definition 3.1(d)]{mochizuki-iut}.

Because of this, the question of determining these curves was raised by Fesenko
at the
\href{https://www.maths.nottingham.ac.uk/personal/ibf/files/symcor.iut.html}{workshop
on inter-universal Teichmüller theory} at the University of Oxford in December
2015. This question then went unanswered until the models in this article were
found.

We will use classical methods, as well as the more recent methods by Shimura
\cite{shimura}, to determine the canonical models of the $4$ arithmetic $(1;
\infty)$-curves in Takeuchi's list. Moreover, we describe the correspondence of
these curves with the classical modular curve $X (1)$, and we give the
arithmetic and modular properties of these canonical models.

A main tool that we will use is that of \defi{\Belyi\ maps}. For a survey that
shows the ubiquity of these maps, we refer to \cite{sv-survey}; they also play
a significant role in \cite{mochizuki-iut}. Expressed briefly, given an
arithmetic group $\PGamma$ commensurable with a triangle group $\PDelta$, the
inclusion of groups $\PGamma \cap \PDelta \subset \PDelta$ gives rise to a
\Belyi\ map $X (\PGamma \cap \PDelta) \to X (\PDelta)$, which is a cover of a
projective line ramified above at most three points. The combinatorics of this
inclusion yield equations for $X (\PGamma)$ via algorithms such as those in
\cite{kmsv}, and the correspondence with the curve $X (\PDelta)$ can be used to
obtain significant arithmetic information regarding the curve $X (\PGamma)$, for
instance properties of its canonical model. These considerations apply in
particular to the classical case $\PDelta = \PGamma (1) = \PSL_2 (\Z)$ in which
we find ourselves.

Section \ref{sec:uniford} introduces the relevant groups and their associated
orders. Section \ref{sec:belyi} gives the resulting \Belyi\ maps, which are
used in Section \ref{sec:canmod} to determine the canonical models and their
arithmetic properties. Finally, the modular interpretation of these models is
described in Section \ref{sec:modular}.

\textbf{Acknowledgments.} I cordially thank Ivan Fesenko for pointing out the
relevance of the curves in this article; without his interest this paper would
not have existed. Further heartfelt thanks go out to John Voight for sharing
his wisdom on dessins, canonical models and conductors.

\section{Uniformizations and orders}\label{sec:uniford}

\setcounter{subsection}{1}

In this section, and in fact throughout the article, our exposition will be
rather condensed; a more comprehensive exposition of these topics is in
\cite{takeuchi-1e, sijsling-thesis, sijsling-1e}.

In \cite[Proposition 3.1, Theorem 3.4]{takeuchi-1e}, Takeuchi describes his
$(1; \infty)$-groups as follows.

\begin{lemma}
  Let $\PGamma \subset \PSL_2 (\R)$ be an arithmetic $(1; \infty)$-group. Then
  the inverse image $\Gamma$ of $\PGamma$ in $\SL_2 (\R)$ is generated by two
  elements $\alpha$ and $\beta$ whose commutator is parabolic and for which the
  traces of $\alpha$, $\beta$ and $\alpha \beta$ are given in Table
  \ref{tab:trtr}. Conversely, these traces determine $\Gamma$ up to $\SL_2
  (\R)$-conjugacy.
\end{lemma}

\begin{table}[h]
  \begin{center}
    \begin{tabular}{c|ccc}
      Case & $\Tr (\alpha)$ & $\Tr (\beta)$ & $\Tr (\alpha \beta)$ \\
      \hline
      I    & $\sqrt{5}$     & $2 \sqrt{5}$  & $5$                  \\
      II   & $\sqrt{6}$     & $2 \sqrt{3}$  & $3 \sqrt{2}$         \\
      III  & $2 \sqrt{2}$   & $2 \sqrt{2}$  & $4$                  \\
      IV   & $3$            & $3$           & $3$                  \\
    \end{tabular}
  \end{center}
  \caption{Trace triples.}
  \label{tab:trtr}
\end{table}

As in \cite[\S 1.2]{sijsling-thesis}, we construct $\alpha$ and $\beta$ by
taking
\begin{equation}
  \alpha =
  \begin{pmatrix}
    \lambda & 0 \\
    0 & \lambda^{-1}
  \end{pmatrix},
  \beta =
  \begin{pmatrix}
    a & b \\
    b & d
  \end{pmatrix} .
\end{equation}
Here $\lambda$ is the largest real solution of $\lambda^{-1} + \lambda = \Tr
(\alpha)$. Having determined $\lambda$, we recover $a$ and $d$ by from $\Tr
(\beta)$ and $\Tr (\alpha \beta)$ by solving a system of linear equations.
Finally, we choose $b$ to be the positive real solution of $a d - b^2 = 1$.

The matrices $\alpha$, $\beta$ thus obtained are far from having entries in
$\Z$. However, in the proof of \cite[Theorem 3.4]{takeuchi-1e} (and see also
his work \cite{takeuchi-char}) Takeuchi shows the following more precise
statement.

\begin{lemma}\label{lem:B}
  Let $S$ be the set $\left\{ 1, \alpha^2, \beta^2, \alpha^2 \beta^2 \right\}$,
  and let $\Gamma^{(2)} = \langle \gamma^2 : \gamma in \Gamma \rangle$.
  \begin{enumerate}
    \item The group $\Gamma^{(2)}$ has signature $(1; \infty^4)$, and the image
      $\PGamma^{(2)}$ of $\Gamma^{(2)}$ in $\PSL_2 (\R)$ is normal in the image
      $\PGamma \subset \PSL_2 (\R)$. The quotient group $\PGamma /
      \PGamma^{(2)}$ is a Klein Vierergruppe generated by $\alpha$ and $\beta$.
    \item The $\Q$-vector space generated by $S$ is a quaternion algebra over
      the rationals.
    \item The $\Z$-module generated by $\Gamma^{(2)}$ is an order $\O$ of
      $B$, and the group $\O^1$ of elements of $\O$ with reduced norm $1$
      contains $\Gamma^{(2)}$.
  \end{enumerate}
\end{lemma}

\begin{remark}
  In \cite[Proposition 3.1]{takeuchi-1e} Takeuchi also gives a list of
  generators of $\Gamma^{(2)}$, so that we can explicitly construct the order
  $\O$ in part (ii) of the Lemma.
\end{remark}

In our cases $\Gamma$ has a cusp, hence so does $\Gamma^{(2)}$. Therefore $\O$
is in fact conjugate to an order in the trivial quaternion algebra $M_2 (\Q)$
over $\Q$. This simplifies our considerations, since the original matrices
$\alpha$ and $\beta$ will have entries in an extension of $\Q$; in case II this
extension is even a biquadratic number field. To find an order in $M_2 (\Q)$ to
which $\O$ is conjugate, we can let $\Gamma$ act on $S$ by left multiplication
and identify the resulting associative algebra with $M_2 (\Q)$ by using the
algorithms in \cite{voight-matrix}.

We have slightly tweaked this approach; instead of taking $\O = \Z [
\Gamma^{(2)} ]$, we take $\O' = \Z [ \Gamma' ]$, where $\Gamma'$ is the largest
subgroup between $\Gamma^{(2)}$ and $\Gamma$ that generates some order in the
algebra $B$.

\begin{proposition}
  In the cases I-IV, the largest groups $\Gamma'$ inbetween $\Gamma^{(2)}$ and
  $\Gamma$ that generate an order in the matrix algebra $B$ in Lemma
  \ref{lem:B} are given in Table \ref{tab:gammap}.
\end{proposition}

\begin{table}[h]
  \begin{center}
    \begin{tabular}{c|c}
      Case & Largest $\Gamma'$ generating an order \\
      \hline
      I    & $\langle \Gamma^{(2)}, \alpha \beta \rangle$ \\
      II   & $\Gamma^{(2)}$ \\
      III  & $\langle \Gamma^{(2)}, \alpha \beta \rangle$ \\
      IV   & $\Gamma$ \\
    \end{tabular}
  \end{center}
  \caption{Groups used to generate orders.}
  \label{tab:gammap}
\end{table}

\begin{proof}
  The quotient $\PGamma / \PGamma^{(2)}$ is a Klein Vierergruppe generated by
  $\alpha$ and $\beta$, so it suffices to see which of $\alpha, \beta, \alpha
  \beta$ can be written as a $\Q$-linear combination of the basis $S$ of $B$,
  which is a straightforward computation as $S$ is also an $\R$-basis of $M_2
  (\R)$.
\end{proof}

The groups $\Gamma'$ are indicated in Table \ref{tab:gammap}. After
trivializing the quaternion algebras $B$ involved, we have conjugated the
corresponding orders $\O$ of $M_2 (\Q)$ into $M_2 (\Z)$ by an \emph{ad hoc}
calculation. In general, one could determine a maximal order of $M_2 (\Q)$
containing $\O$. Such a maximal order is conjugate to $M_2 (\Z)$, so that an
explicit conjugating element can be used to map $\O$ into $M_2 (\Z)$.

Identifying the groups $\Gamma'$ and $\Gamma$ with their images under the
preceding conjugations, we now have reduced to a simpler situation, where we
have an inclusion $\Gamma^{(2)} \subset \Gamma'$ as well as inclusions
\begin{equation}
  \Gamma \supset \Gamma' \subset \Z [\Gamma']^1 = \O^1 \subset \SL_2 (\Z) .
\end{equation}
Proposition \ref{cor:inter} will show that in fact $\Gamma' = \Gamma \cap \SL_2
(\Z)$. Note that it is \emph{not} clear that $\Gamma'$ equals $\Z [\Gamma']^1$;
in fact in case IV the order $\O'$ is a non-Eichler order of index $4$ whose
group of units is of index $2$ in $\SL_2 (\Z)$ and hence properly contains in
$\Gamma$. In all the other cases, Proposition \ref{prop:gpvszgp1} will show
that indeed $\O'^1 = \Gamma'$. In fact, in case IV we still have that $\Z
[\Gamma^{(2)}]^1 = \Gamma^{(2)}$, so that we could have worked with the
original order $\O = \Z [\Gamma^{(2)}]$ defined by $\Gamma^{(2)}$ instead.
Regardless, we will see below that case IV is less complicated than the others.

\section{\Belyi\ maps}\label{sec:belyi}

\subsection{Monodromy triples}

The reason for our using the unit group $\O'^1$ is that it is less complicated
to determine whether a given element $\gamma \in \SL_2 (\Z)$ belongs to this
group; it suffices to see whether $\gamma \in \O'$, which gives a
straightforward additive criterion. In this way, we can quickly determine the
monodromy of the cover (and in fact \Belyi\ map) corresponding to the inclusion
$\O'^1 \subset \SL_2 (\Z)$ by using the standard generators $S$, $T$ and $(S
T)^{-1}$ of $\SL_2 (\Z)$ of order $2, 3, \infty$ and using the methods of
\cite[\S 3]{kmsv}.

\begin{remark}
  It is also possible to work directly with one of the groups $\Gamma^{(2)}$,
  $\Gamma'$ or $\Gamma$ by using Dirichlet domains. For details, we again refer
  to \cite[\S 3]{kmsv}, especially Algorithms 3.8 and 3.14.
\end{remark}

In case IV the group $\Gamma$ is itself a subgroup of $\SL_2 (\Z)$. Considering
the possible ramification then shows that it is of index $6$, and that the
cover $X (\Gamma) \to X (1)$ has ramification type $(3^2), (2^3), (6^1)$ above
$j = 0, 1728, \infty$ respectively. Geometrically, this is nothing but the
quotient of an elliptic curve with $j$-invariant $0$ by its automorphism group;
in the next section, we determine the correct model of this cover over its
canonical field of definition $\Q$.

In the other cases, we have to put in more effort. The permutation triples
$(\sigma_0, \sigma_1, \sigma_{\infty})$ that describe the monodromy of the
covers defined by the inclusion $\O'^1 \subset \SL_2 (\Z)$ above $j = 0, 1728,
\infty$ are given in Table \ref{tab:monod}. We have followed the convention
$\sigma_{\infty} \sigma_1 \sigma_0 = 1$. Note that as mentioned above we get a
subgroup of index $2$, not $6$, in case IV. However, we have the following.

\begin{table}[h]
  \begin{center}
    \begin{tabular}{c|ccc}
      Case & & & Monodromy \\
      \hline
           & $\sigma_0$ & = & $(1, 3, 8)(2, 7, 4)(5, 11, 9)(6, 12, 10)$ \\
      I    & $\sigma_1$ & = & $(1, 2)(3, 5)(4, 6)(7, 10)(8, 9)(11, 12)$ \\
           & $\sigma_{\infty}$ & = & $(1, 4, 10, 2, 8, 11, 6, 7, 12, 5)(3, 9)$ \\
      \hline
           & $\sigma_0$ & = & $(1, 3, 8)(2, 7, 4)(5, 14, 9)(6, 15, 10)$ \\
           & & & $(11, 16, 12)(13, 19, 17)(18, 23, 20)(21, 22, 24)$ \\
      II   & $\sigma_1$ & = & $(1, 2)(3, 5)(4, 6)(7, 11)(8, 12)(9, 13)$ \\
           & & & $(10, 14)(15, 17)(16, 18)(19, 21)(20, 22)(23, 24)$ \\
           & $\sigma_{\infty}$ & = & $(1, 4, 10, 5)(2, 8, 16, 20, 21, 13, 14, 15, 19, 24, 18, 11)$ \\
           & & & $(3, 9, 17, 6, 7, 12)(22, 23)$ \\
      \hline
           & $\sigma_0$ & = & $(1, 3, 8)(2, 7, 4)(5, 11, 9)(6, 12, 10)$ \\
      III  & $\sigma_1$ & = & $(1, 2)(3, 5)(4, 6)(7, 11)(8, 10)(9, 12)$ \\
           & $\sigma_{\infty}$ & = & $(1, 4, 10, 3, 9, 6, 7, 5)(2, 8, 12, 11)$ \\
      \hline
           & $\sigma_0$ & = & $(1)$ \\
      IV   & $\sigma_1$ & = & $(1, 2)$ \\
           & $\sigma_{\infty}$ & = & $(1, 2)$ \\
      \hline
    \end{tabular}
  \end{center}
  \caption{Monodromy generators for $\O'^1 \subset \SL_2 (\Z)$}
  \label{tab:monod}
\end{table}

\begin{proposition}\label{prop:gpvszgp1}
  In the cases I-III, we have $\Z [\Gamma']^1 = \Gamma'$ in Table
  \ref{tab:gammap}.
\end{proposition}

\begin{proof}
  This follows from Table \ref{tab:monod} since clearly $\Gamma' \subset \Z
  [\Gamma']^1$ and the index of $\Z [\Gamma']^1$, as read off from that table,
  coincides with the index of $\Gamma'$, as determined by the covolume of this
  group.
\end{proof}

\begin{corollary}\label{cor:inter}
  In all cases, we have $\Gamma' = \Gamma \cap \SL_2 (\Z)$.
\end{corollary}

\begin{proof}
  The inclusion $\Gamma' \subset \Gamma \cap \SL_2 (\Z)$ clearly holds, and if
  the intersection were larger, then the corresponding group would generate an
  order in $M_2 (\Z)$ larger than that generated by $\Gamma'$, a
  contradiction with the definition of the latter group.
\end{proof}

We have to determine the \Belyi\ maps corresponding to the monodromy triples in
Table \ref{tab:monod}. These maps are of sufficiently low degree to be
accessible directly by general methods such as those developed in \cite{kmsv},
and the functionality under development in \cite{kmssv} quickly returns
explicit formulas for them. Still, we indicate a useful technique in the
calculation of covers that facilitates these computations.

\subsection{Decomposing covers}

Let $\fbar : \Zbar \to \Xbar$ be a ramified cover of degree $d$ of a known base
space $\Xbar$ that is the projective completion of an unramified cover $f : Z
\to X$. We wish to determine whether $f$ can be written as a non-trivial
composition $Z \to Y \to X$ (which will imply a similar statement for $\fbar$),
and if so to describe the maps $Z \to Y$ and $Y \to X$ combinatorially. When
studying \Belyi\ maps, we are of course considering $\Xbar = \P^1$, $X = \P^1 -
\left\{ 0, 1, \infty \right\}$.

The cover $f$ itself corresponds to a homomorphism $\phi : \pi_1 (X, x) \to
S_d$, where $x$ is a chosen base point of $X$. Let $G$ be the image of $\phi$,
and let $H$ be the stabilizer of $1 \in \left\{ 1, \dots, d \right\}$. Then
$G$ is canonically isomorphic to a quotient of $\pi_1 (X, x)$, and in this way
the set of cosets $G / H$ becomes a $\pi_1 (X, x)$-set. After having chosen a
presentation of $\pi_1 (X, x)$, its generators are mapped by $\phi$ to
elements of $S_d$, which group we can identify with $\Sym (G / H)$. In the
case where such a generator represents a simple loop around a branch point in
the completion $\Xbar$ of $X$, its image under $\phi$ describes the
local monodromy.

Via the dictionary of covering theory \cite{kmsv, sijsling-belyi}, a
factorization of $f$ as $Z \to Y \to X$ corresponds to a subgroup $K$
inbetween $H$ and $K$. In our cases $X = \P^1 - \left\{ 0, 1, \infty
\right\}$, and we have determined the representation $\pi_1 (X, x) \to S_d$ up
to conjugacy in Table \ref{tab:monod}. Since the groups involved are very
small, we can calculate the lattice of subgroups of $G$ that contain $H$. In
practice we do not merely find a factorization in one step, but a succession
of inclusions $H \subset K \subset K' \subset \dots \subset G$. We wish to
describe the resulting covers, for which we may without loss of generality
consider the case $H \subset K \subset G$ of a single intermediary subgroup.

For a start, the inclusion $K \subset G$ corresponds to the $\pi_1 (X, x)$-set
$G / K$. This means that we can reuse the generators of $\pi_1 (X, x)$ on this
smaller set of cosets to describe the monodromy, which furnishes a
combinatorial description of the cover $Y \to X$.

The cover corresponding to the inclusion $H \subset K$ can be more difficult to
describe, since in this case we have to describe a cover of $Y$; in particular,
the fundamental group used changes from $\pi_1 (X, x)$ to $\pi_1 (Y, y)$, where
$y$ is some element of $Y$ that is in the fiber of $Y$ over $x$. Using the
methods of \cite{kmsv}, it is possible to calculate a presentation of $\pi_1
(Y, y)$ in terms of the elements of $\pi_1 (X, x)$, and with it the
representation of this group on the set of cosets $K / H$. In the coming lemma,
we determine some weaker invariants, namely the local monodromy around the
points of $Y$. In all of the cases under consideration in this article, this
suffices for the calculation of the covers involved. In stating it, we done the
projective completion of a curve $X$ by $\Xbar$.

\begin{lemma}\label{lem:decomp}
  Let $f : Z \to X$ be a cover described by the monodromy morphism $\phi :
  \pi_1 (X, x) \to G \subset S_d$, and let $K$ be a group inbetween $H = \Stab
  (1)$ and $G$. Let $Z \to Y \to X$ be the resulting factorization of $f$. Let
  $\gamma$ be an element of $\pi_1 (X, x)$ that represents a simple loop around
  an element $\xbar$ of $\Xbar - X$, and let $\sigma = \phi (\gamma)$ be the
  corresponding local monodromy. Then the monodromy of the cover $Z \to Y$
  around the points $\ybar$ of $\Ybar$ over $\xbar$ can be described as
  follows.
  \begin{enumerate}
    \item The points $\ybar$ correspond to the equivalence classes in $G / K$
      under left multiplication by $\sigma$.
    \item Let $c \in G / K$ be a coset representing a point as in (i). Define
      $\tau \in K$ by
      \begin{equation}
        \tau = (c^{-1} \sigma c)^e = c^{-1} \sigma^e c
      \end{equation}
      where $e$ is the length of the cycle obtain by multiplying with $\sigma$.
      Then the monodromy around the point $\ybar$ is described by $\tau$.
  \end{enumerate}
\end{lemma}

\begin{proof}
  This follows by considering the loop around $\xbar$ in $\pi_1 (X, x)$. It
  lifts to various segments of loops on $Y$, which need not again be loops.
  However, the smallest powers described $\tau$ described in the Lemma
  correspond to simple loops on $Y$ with various centers in the fiber of
  $\Ybar$ over $\xbar$. Since all loops on $Y$ around a point in such a fiber
  come from exactly one such orbit under powering, our statement is proved.
\end{proof}

We can now determine the covers that we need by the decomposition procedure
described above.

\begin{lemma}\label{lem:case2} In case II, the \Belyi\ map $X (\Gamma) \to X
  (1)$ is described by the map
  \begin{equation}
    (x, y) \mapsto
    6912 \frac{(2 x^3 - 6 x^2 - 1)^3}{(x - 2)^6 (x + 1)^3 x^2 (x - 3)}
  \end{equation}
  from the curve $y^2 = x (x - 1) (3 x - 2) (3 x + 1)$.
\end{lemma}

\begin{proof}
  In this case the monodromy group has cardinality $288$ and trivial center. We
  use Lemma \ref{lem:decomp}. There multiple ways in which we can use the
  subgroup lattice between $H$ and $G$. We have chosen a particularly simple
  one; between $H$ and $G$ there is a group that contains $H$ of index $2$ and
  that gives rise to a cover of $\P^1 - \left\{ 0, 1, \infty \right\}$ with
  permutation triple
  \begin{equation}\label{eq:case2H}
    \begin{aligned}
      \sigma_1 & = (1, 2)(3, 4)(5, 10)(6, 9)(7, 8)(11, 12) \\
      \sigma_0 & = (1, 3, 6)(2, 5, 4)(7, 9, 10)(8, 11, 12) \\
      \sigma_{\infty} & = (1, 4)(2, 6, 7, 12, 8, 10)(3, 5, 9)
    \end{aligned}
  \end{equation}
  Riemann-Hurwitz shows that the resulting cover has genus $0$. Moreover, since
  all ramification above the elliptic points of index $2$ and $3$ has been
  absorbed, we see that the corresponding curve has signature $(0; \infty^4)$,
  where the cusps correspond to the $4$ elements above $\infty$. We can then
  recover $X (\Gamma^{(2)})$ by taking a degree $2$ cover that
  ramifies above these $4$ points. So we instead consider the cover $Z \to X$
  associated to \eqref{eq:case2H} and the associated groups $G$ and $H$.

  Using the subgroup lattice shows that once more there is an intermediate
  subgroup $K$ between $G$ and $H$, generated by the elements
  \begin{equation}
    \begin{split}
      (1, 4)(2, 6, 7, 12, 8, 10)(3, 5, 9), \\
      (1, 11, 4)(2, 3, 12)(5, 6, 8)(7, 9, 10), \\
      (2, 8, 7)(3, 5, 9)(6, 10, 12) .
    \end{split}
  \end{equation}
  It contains $H$ of index $3$. We get a factorization
  $Z \to Y \to X$. Considering cosets of $K$ shows that $Y \to X$ is described
  by the triple
  \begin{equation}
    \begin{aligned}
      \sigma_1 & = (1, 2)(3, 4) \\
      \sigma_0 & = (1, 2, 3) \\
      \sigma_{\infty} & = (2, 3, 4)
    \end{aligned}
  \end{equation}
  This gives the cover
  \begin{equation}\label{eq:case2K}
    x \mapsto 6912 \frac{x^3 (x + 2)}{4 x - 1}
  \end{equation}
  of the $j$-line.

  The cover $Z \to Y$ ramifies over the points $x = -2$ in the fiber over $j =
  0$ and the points $x = 1/4, \infty$ in the fiber over $j = \infty$. Applying
  Lemma \ref{lem:decomp} shows that the ramification over these points is given
  by $(3^1, 2^1 1^1, 2^1 1^1)$. The \Belyi\ map $b : x \mapsto (-27/4) x^2 (x -
  1)$ ramifies in this way over $0$, $1$ and $\infty$, so by moving these
  branch points we get our cover $Z \to Y$, which we can take to be $(-8 b -
  1)/(4 b - 4)$. We obtain $Z \to X$ by substituting this latter map in the
  former map \eqref{eq:case2K} and then polishing by subsituting $x/3$ for
  $x$.
\end{proof}

Similarly, we obtain:

\begin{lemma}\label{lem:case1} In case I, the \Belyi\ map $X (\Gamma) \to X
  (1)$ is described by the map
  \begin{equation}\label{eq:case1}
    (x, y) \mapsto
    \frac{- (x^2 - 10 x + 5)^3}{x}
  \end{equation}
  from the curve $y^2 = x (x^2 - 22 x + 125)$.
\end{lemma}

%

\begin{lemma}\label{lem:case3} In case III, the \Belyi\ map $X (\Gamma) \to X
  (1)$ is described by the map
  \begin{equation}
    (x, y) \mapsto 256 \frac{(x^2 + 1)^3}{x^4}
  \end{equation}
  from the curve $y^2 = x (4 x^2 + 1)$.
\end{lemma}

%

\section{Canonical models}\label{sec:canmod}

We now determine canonical models for our $(1; \infty)$-curves, using two
different methods.

\subsection{First method: $q$-expansions}

The first method is elementary and arguably still the most insightful; we
choose the defining equations that we encounter in such a way that the
resulting $q$-expansions of the coordinates are rational.

To illustrate this method, we consider case II, where we first try the model
furnished in Lemma \ref{lem:case2}, namely $y^2 = x (x + 1) (x - 2) (x - 3)$.
In this case the fiber of the \Belyi\ map above $\infty$ has $4$ elements with
distinct ramification indices. All of these will therefore give rise to a
branch that gives a rational $q$-expansion of $x$. However, we also need to
consider the square root that we have to draw when determining $y$, and this
only gives a rational $q$-expansion for that coordinate if we switch to the
quadratic twist by $-1$. This fully determines the canonical model over $\Q$ of
the curve $X (\Gamma^{(2)})$. It is independent of the particular point over
$\infty$ that is chosen.

Since $\Gamma^{(2)} \neq \Gamma$, it still remains to determine canonical model
of $X (\Gamma)$ in this case. This is obtained by dividing out the $2$-torsion
of the Jacobian, which in fact leads to an isomorphic curve. Case IV
corresponds to a subgroup of $\SL_2 (\Z)$, so we do not have to take any
further isogeny, while in case I we can recover a model of $X (\Gamma)$ from
the given one for $X (\Gamma')$ be taking an isogeny with kernel $(0, 0)$.

The case III leads to a subtlety, due to an automorphism of the cover obtained
from the group $\Gamma'$ (which is the map $x \mapsto -x$ in Lemma
\ref{lem:case3}). In this case there is no common quadratic twist that makes
both branches over $\infty$ have rational $q$-expansions. Either choice of
branch leads to a model of $X (\Gamma')$ over $\Q$. Neither of them can be
called canonical over $\Q$, however, but only over $\Q (i)$ where either chosen
model over $\Q$ leads to all branches being rational. On the other hand, the
model of the codomain $X (\Gamma)$ of the resulting $2$-isogeny (which again
has kernel $(0, 0)$) does not depend on the choice of branch, and we get
rational $q$-expansions at the single cusp of this curve. Therefore also in
this case we obtain a canonical model for $X (\Gamma)$ over $\Q$, even though
$X (\Gamma')$ admits no such model.

Table \ref{tab:canmod} summarizes the canonical models found; the parabolic
point corresponds to the point at infinity in all cases. Besides a reference to
the LMFDB \cite{lmfdb}, where more detailed information on these curves can be
found, we have also included their $j$-invariant and Faltings height
$h_{\textrm{Falt}}$. Table \ref{tab:exps} gives some of the resulting
$q$-expansions of the coordinates $x$ and $y$.

\begin{table}[h]
  \begin{center}
    \begin{tabular}{c|cccc}
      Case & Curve & LMFDB label & $j$-invariant & $h_{\textrm{Falt}}$ \\
      \hline
      I    & $y^2 = x^3 - 44 x^2 - 16 x$ & \href{http://www.lmfdb.org/EllipticCurve/Q/20/a/1}{20.a1} & $2^{14} 31^3/ 5^3$ & $3.814$ \\
      II   & $y^2 = x^3 - 4 x^2 - 384 x - 2304$ & \href{http://www.lmfdb.org/EllipticCurve/Q/24/a/3}{24.a3} & $2^2 73^3 / 3^4$ & $3.152$ \\
      III  & $y^2 = x (x^2 - 256)$ & \href{http://www.lmfdb.org/EllipticCurve/Q/32/a/4}{32.a4} & $1728$ & $-1.311$ \\
      IV   & $y^2 = x^3 - 1728$ & \href{http://www.lmfdb.org/EllipticCurve/Q/36/a/3}{36.a3} & $0$ & $-1.321$
    \end{tabular}
  \end{center}
  \caption{Canonical models}
  \label{tab:canmod}
\end{table}

\begin{table}[h]
  \begin{center}
    \begin{tabular}{c|c}
      Case & Expansions of $x$ and $y$ \\
      \hline
      I    & $x = q^{-1/5} + 16 + 134 q^{1/5} + 760 q^{2/5} + 3345 q^{3/5} + 12256 q^{4/5} + \dots$ \\
           & $y = q^{-3/10} + 2 q^{-1/10} - 129 q^{1/10} - 1778 q^{3/10} - 13725 q^{1/2} - \dots$ \\
      II   & $x = q^{-1/6} + 2 + 79 q^{1/6} + 352 q^{1/3} + 1431 q^{1/2} + 4160 q^{2/3} + \dots$ \\
           & $y = q^{-1/4} + q^{-1/12} - 76 q^{1/12} - 778 q^{1/4} - 5224 q^{5/12} + \dots$ \\
      III  & $x = q^{-1/4} + 52 q^{1/4} + 834 q^{3/4} + 4760 q^{5/4} + 24703 q^{7/4} + \dots$ \\
           & $y = q^{-3/8} - 50 q^{1/8} - 2599 q^{5/8} - 29154 q^{9/8} - 238728 q^{13/8} + \dots$ \\
      IV   & $x = q^{-1/3} + 824 q^{2/3} - 613348 q^{5/3} + 831470016 q^{8/3} + \dots$ \\
           & $y = q^{-1/2} + 372 q^{1/2} + 29250 q^{3/2} - 134120 q^{5/2} + \dots$
    \end{tabular}
  \end{center}
  \caption{$q$-expansions}
  \label{tab:exps}
\end{table}

The curves that we found have minimal conductor among all their possible
quadratic twists over $\Q$. Note that while case III gives a curve that is
isomorphic to $X_0 (32)$ over $\Q$, the latter curve is in fact a $(1;
\infty^8)$-curve. The corresponding groups are commensurable, with neither
being a subgroup of the other up to conjugacy.

\subsection{Second approach: Shimura's viewpoint}

A second approach to obtain canonical models is to use work by Shimura
\cite{deligne, shimura} that gives a more general description of canonical
models and their properties. In \cite{sijsling-1e}, this served as an essential
tool to determine canonical models in the presence of a non-split quaternion
algebra.

In this approach we further enlarge the orders $\O'$ from Section
\ref{sec:uniford} to orders $\O'' \supset \O'$ with $\O''^1 = \O'^1$ but with
the property that the norm surjects to as large a subset of $\Z$ as possible.
Shimura then describes the canonical field of definition of $X (\O''^1)$ and
indicates how the traces of Frobenius of the canonical model over this field
can be determined. We briefly sketch the results here.

In cases I and II the overorder $\O''$ can be chosen in such a way that the
norm surjects to $\Z$. In this case the canonical field of definition equals
$\Q$. The corresponding curve has good reduction outside the primes that divide
the index of $\O''$ in $M_2 (\Z)$, and its traces of Frobenius can be
determined using Shimura congruence, as in \cite[Chapter 4]{sijsling-thesis}.
In combination with the known $j$-invariant, these data suffice to determine
the canonical model completely. We checked that the results obtained agree with
the first approach.

In case III we can at most obtain an order $\O''$ whose norm map attains all
primes that are congruent to $1$ modulo $4$. Shimura's results then show that
the canonical field of definition of the curve $X (\O'^1) =  X (\langle
\Gamma^{(2)}, \alpha \beta \rangle)$ is $\Q (i)$, which is in line with the
subtlety encountered when using the first method above. The methods from
\cite{sijsling-thesis} can again be used to compute the traces of Frobenius of
Deligne's non-connected model at all primes, and once more the results agree
with the more elementary method above.

Case IV is quite surprising. In this case $\PGamma$ is the image of the
commutator subgroup of $\SL_2 (\Z)$ in $\PSL_2 (\R)$. Indeed, like the former
group, the latter is abelian of index $6$ in $\PSL_2 (\R)$. By the universal
property of the commutator subgroup, these groups therefore have to coincide.

The commutator subgroup of $\SL_2 (\Z)$ is described in \cite[Remark
3.9]{conrad-sl2}. Let
\begin{equation}
  \gamma_1 = \begin{pmatrix}
    1 & 1 \\
    1 & 2
  \end{pmatrix} ,
  \gamma_2 = \begin{pmatrix}
    1 & -1 \\
    -1 & 2
  \end{pmatrix}
\end{equation}
be generators of the commutator subgroup of $\SL_2 (\Z)$. Recall that the order
$\O' = \Z [\Gamma]$ had the property $\Gamma \subsetneq \O'^1$ and is therefore
too large to study $\Gamma$. We can instead consider the order $\O = \Z
[\Gamma^{(2)}]$. Then
\begin{equation}
  \Gamma = \langle \O^1, \gamma_1, \gamma_2 \rangle .
\end{equation}
It turns out that $\O$ can be enlarged to an order of index $36$ in $M_2 (\Z)$
to which none of $\gamma_1, \gamma_2, \gamma_1 \gamma_2$ belongs. Moreover, the
groups $\langle \Gamma^{(2)}, \alpha \rangle$, $\langle \Gamma^{(2)}, \beta
\rangle$, $\langle \Gamma^{(2)}, \alpha \beta \rangle$ all generate additional
orders of index $36$ to which exactly one of these commutators belongs. These
orders are distinct, but have the same index in $M_2 (\Z)$ and give rise to
curves with canonical field of definition $\Q$ whose Jacobians are all
isogenous.

\subsection{Orders obtained}

Bases of the orders $\O''$ can be found by running the code at
\cite{sijsling-1inf-code}. None of these orders is Eichler, and in all cases
the index in $M_2 (\Z)$ equals the corresponding conductor in Table
\ref{tab:canmod}. In the next section we will construct these orders $\O''$ as
intersections.

\section{Modular interpretations}\label{sec:modular}

The curves $X (\Gamma)$ also parametrize the points of certain moduli functors,
which we describe in this section. A more advanced general exposition is given
in \cite{deligne-rapoport}.

\subsection{Over $\C$}

Consider pairs
\begin{equation}\label{eq:modpair}
  (E, (\lambda_1, \lambda_2))
\end{equation}
where $E$ is an elliptic curve over $\C$ and where $(\lambda_1, \lambda_2)$ is
an ordered basis of the homology group $H_1 (E, \Z)$, which we assume to be
\defi{positively oriented} in the sense that $\im (\lambda_1 / \lambda_2) > 0$
in the one-dimensional $\C$-vector space $V = H_1 (E, \Z) \otimes \R = H^0 (E,
\omega_E)^{\vee}$ that is the universal cover of $E$.

The group $\SL_2 (\Z)$ acts on the left on the set of objects
\eqref{eq:modpair} as follows: if $\gamma' = \smat{a}{b}{c}{d} \in \SL_2 (\Z)$,
then
\begin{equation}\label{eq:action}
  \gamma' (E, (\lambda_1, \lambda_2))
  = (E, (a \lambda_1 + b \lambda_2, c \lambda_1 + d \lambda_2)) .
\end{equation}
Let $\Gamma'$ be a subgroup of $\SL_2 (\Z)$ (not necessarily congruence). We say
that two pairs $(E, (\lambda_1, \lambda_2))$, $(E', (\lambda'_1, \lambda'_2))$
are $\Gamma'$-equivalent if and only if there exists an isomorphism $\phi : E
\to E'$ such that
\begin{equation}\label{eq:equiv}
  (E', (\phi_* (\lambda_1), \phi_* (\lambda_2)))
  = \gamma' (E', (\lambda'_1, \lambda'_2))
\end{equation}
for some $\gamma'$ in $\Gamma'$.

\begin{proposition}
  The points of the quotient space $X (\Gamma')$ bijectively correspond to the
  $\Gamma'$-equivalence classes of pairs \eqref{eq:modpair}.
\end{proposition}

\begin{proof}
  Given a pair as in \eqref{eq:modpair}, there exists a unique $\tau \in \H$ such
  that
  \begin{equation}
    (E, (\lambda_1, \lambda_2)) \cong (E_{\tau}, (\tau, 1)) .
  \end{equation}
  Here $E_{\tau} = V / \Lambda_{\tau}$. The action of $\Gamma'$ is then given via
  the usual action of $\PSL_2 (\R)$ on the upper half plane $\H$. Moreover, two
  pairs $(E_{\tau}, (\tau, 1))$, $(E_{\tau'}, (\tau', 1))$ are equivalent under
  if and only if $\tau$ and $\tau'$ are related by this action.
\end{proof}

In the cases I-III, the group $\Gamma$ is not itself a subgroup of $\SL_2
(\Z)$, and we first restrict to $\Gamma' = \Gamma \cap \SL_2 (\Z)$. In all of
these cases, the group $\Gamma$ can be obtained from $\Gamma'$ by adjoining
elements $w \in M_2 (\Z)$ that are not in $\SL_2 (\Z)$ but whose square is a
still scalar. These elements $w$ normalize the order $\O'$ defined by
$\Gamma'$.

The modular interpretation of these involutions in terms of the pairs
\eqref{eq:modpair} is as follows. Let $s$ be the square of $w =
\smat{a}{b}{c}{d}$, considered as a scalar, let $w' = w^{-1}$, and let
$\Lambda^w = w' (\Lambda)$ be the overlattice of $\Lambda$ defined by $w'$.
Then $\Lambda^w$ contains $\Lambda$ of index $s$. We define
\begin{equation}
  E^w = V / \Lambda^w
\end{equation}
and define
\begin{equation}
  (\lambda_1^w, \lambda_2^w) = w' (\lambda_1, \lambda_2) = (a \lambda_1 + b
  \lambda_2, c \lambda_1 + d \lambda_2)
\end{equation}
as in \eqref{eq:action}. Then the association
\begin{equation}
  (E, (\lambda_1, \lambda_2)) \mapsto (E^w, (\lambda_1^w, \lambda_2^w))
\end{equation}
gives rise to a well-defined map on $\Gamma'$-orbits because $w^{-1} \Gamma' w
= \Gamma'$.

Restricting to pairs of the form $(E_{\tau}, (\tau, 1))$ as above, we see the
following.

\begin{proposition}\label{prop:XGamma}
  Suppose that $\Gamma = \langle \Gamma', S \rangle$, where $S$ is a set of
  involutions $w$ whose scalar is a square and that satisfy $w \Gamma' w^{-1} =
  \Gamma'$. Then the points of the quotient space $X (\Gamma)$ bijectively
  correspond to the equivalence classes of pairs \eqref{eq:modpair}, up to the
  additional identifications generated by
  \begin{equation}
    (E, (\lambda_1, \lambda_2)) \sim (E^w, (\lambda_1^w, \lambda_2^w))
  \end{equation}
  for $w \in S$.
\end{proposition}

\subsection{Over general base fields}\label{sec:modgen}

Now suppose that $\Gamma' \subset \SL_2 (\Z)$ is a congruence subgroup, a
situation that applies to our examples. Then another description is possible,
which lends itself to generalization to arbitrary fields. Suppose that the
full-level subgroup $\Gamma (N)$ is a subgroup of $\Gamma'$, and let
$\Gammabar'$ be the image of $\Gamma'$ in $\SL_2 (\Z / N \Z)$. Consider pairs
\begin{equation}\label{eq:modpairgen}
  (E, (\lambdabar_1, \lambdabar_2))
\end{equation}
where $E$ is an elliptic curve and where $(\lambdabar_1, \lambdabar_2)$ is
an ordered basis of the $N$-torsion group $E [N]$.

The group $\SL_2 (\Z / N \Z)$ acts on the left on the set of objects
\eqref{eq:modpairgen} as follows: if $\gamma' =
\smat{a}{b}{c}{d} \in \SL_2 (\Z / N \Z)$, then
\begin{equation}\label{eq:actiongen}
  \gamma' (E, (\lambdabar_1, \lambdabar_2))
  = (E, (a \lambdabar_1 + b \lambdabar_2,
         c \lambdabar_1 + d \lambdabar_2)) .
\end{equation}
Via this action, we also get an action of $\Gamma'$, which projects to $\SL_2
(\Z / N \Z)$. We say that two pairs $(E, (\lambdabar_1, \lambdabar_2))$, $(E',
(\lambdabar'_1, \lambdabar'_2))$ are $\Gamma'$-equivalent if and only if there
exists an isomorphism $\phi : E \to E'$ such that
\begin{equation}\label{eq:equivgen}
  (E', (\phi (\lambdabar_1), \phi (\lambdabar_2)))
  = \gamma' (E', (\lambdabar'_1, \lambdabar'_2))
\end{equation}
for some $\gamma'$ in $\Gamma'$. Then we again have the following.

\begin{proposition}
  The points of the algebraic curve $X (\Gamma')$ over an algebraically closed
  field $k$ bijectively correspond to the $\Gamma'$-equivalence classes of
  pairs \eqref{eq:modpairgen} over $k$.
\end{proposition}

Note that for a non-algebraically closed field $k$, such an equivalence class
can define a $k$-rational point of $X (\Gamma')$ without its constituents being
$k$-rational; if $k$ is perfect, then it suffices that the class is closed
under conjugation by the absolute Galois group of $k$.

We can also give a description of the involution defined by an element $w \in
\Gamma$ with integral scalar square $s$ as above. We define $w'' = N w^{-1} \in
M_2 (\Z / N \Z)$. Say $w'' = \smat{a}{b}{c}{d}$. Then defining
\begin{equation}
  w'' (\lambdabar_1, \lambdabar_2)
  = (a \lambdabar_1 + b \lambdabar_2,
     c \lambdabar_1 + d \lambdabar_2) ,
\end{equation}
we let
\begin{equation}\label{eq:ewgen}
  E^w = E / \langle w'' ( \lambdabar_1, \lambdabar_2 ) \rangle .
\end{equation}
Let $s$ be the denominator of $w^{-1}$. We choose a basis
$(\widetilde{\lambdabar}_1, \widetilde{\lambdabar}_2)$ of $E [ N s ]$ in such a
way that $s \widetilde{\lambdabar}_i = \lambdabar_i$. Let $w''' = s w^{-1}$ and
define
\begin{equation}
  (\lambdabar_1^w, \lambdabar_2^w) = w''' (\widetilde{\lambdabar}_1,
  \widetilde{\lambdabar}_2) .
\end{equation}
which is an element of $E^w [N]$. Then the association
\begin{equation}
  (E, (\lambdabar_1, \lambdabar_2)) \mapsto
  (E^w, (\lambdabar_1^w, \lambdabar_2^w))
\end{equation}
gives rise to a well-defined map on $\Gamma'$-orbits, and gives the requested
modular description. The reader will readily formulate the analogue of
Proposition \ref{prop:XGamma}. Again, over perfect fields $k$ the rational
points of $X (\Gamma)$ parametrize equivalence classes that are closed under
Galois conjugation. 

In our concrete examples, we can give slightly more elegant prime-by-prime
descriptions of the moduli problems, essentially because when $\Gamma' =
\Gamma_0 (N)$ we obtain the classical description in terms of torsion subgroups
(instead of bases). We now proceed to give these descriptions, up to conjugacy.

\subsection{Case I}

In this case the order $\O''$ from the end of the previous section is the
intersection of the index $4$ order
\begin{equation}
  \O_4 = \langle
  \smat{1}{0}{0}{1},
  \smat{1}{0}{0}{-1},
  \smat{0}{1}{1}{1},
  \smat{0}{1}{-1}{1}
  \rangle
\end{equation}
and the index $5$ order
\begin{equation}
  \O_5 = \langle
  \smat{1}{0}{0}{0},
  \smat{0}{1}{0}{0},
  \smat{0}{0}{5}{0},
  \smat{0}{0}{0}{1}
  \rangle .
\end{equation}
We obtain a congruence subgroup of level $10$.

The corresponding modular description admits the following simplification: the
curve $X (\Gamma')$ parametrizes equivalence classes of triples
\begin{equation}\label{eq:modtrip}
  (E, H_5, (\lambdabar_1, \lambdabar_2))
\end{equation}
where $E$ is an elliptic curve, $H_5$ is a subgroup of $E [5]$ of order $5$,
and $(\lambdabar_1, \lambdabar_2)$ is a basis of $E [2]$. Two such triples $(E,
H_5, (\lambdabar_1, \lambdabar_2))$ and $(E', H'_5, (\lambdabar'_1,
\lambdabar'_2))$ are equivalent if and only there exists an isomorphism $\phi :
E \to E'$ that maps $H_5$ into $H'_5$ and such that $\phi (\lambdabar_1,
\lambdabar_2) = \alpha (\lambdabar'_1, \lambdabar'_2)$, where $\alpha'$ is a
power of $\smat{0}{1}{1}{1} \in \SL_2 (\F_2)$.

There is a single involution defined by the element
\begin{equation}
  w_5 = \smat{0}{-1}{5}{0}
\end{equation}
of determinant $5$. In terms of the triples \eqref{eq:modtrip}, its action can
be described by
\begin{equation}
  w_5 (E, H_5, (\lambdabar_1, \lambdabar_2))
  = (E', H'_5, (\lambdabar'_1, \lambdabar'_2)) ,
\end{equation}
where $E'$ is the quotient of $E$ by $H_5$, $H'_5$ is the image of $E [5]$ on
$E'$, and $(\lambdabar'_1, \lambdabar'_2)$ is the image of $(\lambdabar_2,
\lambdabar_1)$ on $E'$.

\subsection{Case II}

In this case the order $\O''$ from the end of the previous section is the
intersection of the index $8$ order
\begin{equation}
  \O_8 = \langle
  \smat{1}{0}{0}{1},
  \smat{1}{0}{0}{-1},
  \smat{0}{0}{2}{0},
  \smat{0}{2}{1}{0}
  \rangle
\end{equation}
and the index $3$ order
\begin{equation}
  \O_3 = \langle
  \smat{1}{0}{0}{0},
  \smat{0}{1}{0}{0},
  \smat{0}{0}{3}{0},
  \smat{0}{0}{0}{1}
  \rangle .
\end{equation}
We obtain a congruence subgroup of level $12$. Its elements parametrize triples
\begin{equation}
  (E, H_3, (\lambdabar_1, \lambdabar_2)) ,
\end{equation}
where this time $H_3$ is a subgroup of order $3$ of $H_3$, and where
$(\lambdabar_1, \lambdabar_2)$ is a basis of $E [4]$ up to the equivalence
defined by the elements of the order in $\SL_2 (\Z / 4 \Z)$ determined by the
elements in the basis of $\O_8$.

There are three involutions, defined by the elements
\begin{equation}
  \begin{aligned}
    w_2 & = \smat{-2}{-2}{3}{2}, \\
    w_3 & = \smat{-3}{-4}{3}{3}, \\
    w_6 & = \smat{0}{2}{-3}{0}, \\
  \end{aligned}
\end{equation}
of determinant $2$, $3$, $6$ respectively. Note that these elements only
commute modulo the group $\O''^1$.

The action of the involution $w_2$ is given by
\begin{equation}
  w_2 (E, H_3, (\lambdabar_1, \lambdabar_2))
  = (E', H'_3, (\lambdabar'_1, \lambdabar'_2)) ,
\end{equation}
where $E'$ is the quotient of $E$ by $2 \lambdabar_1$, where $H'_3$ is the
image of $H_3$ on $E'$, and where $(\lambdabar'_1, \lambdabar'_2)$ is some
point on $E'$ such that $(\lambdabar'_1, 2 \lambdabar'_2)$ equals the image of
$(\lambdabar_2, -\lambdabar_1)$ on $E'$.

The action of the involution $w_2$ is given by
\begin{equation}
  w_3 (E, H_3, (\lambdabar_1, \lambdabar_2))
  = (E', H'_3, (\lambdabar'_1, \lambdabar'_2)) ,
\end{equation}
where $E'$ is the quotient of $E$ by $H_3$, where $H'_3$ is the image of $E
[3]$ on $E'$, and where $(\lambdabar'_1, \lambdabar'_2)$ is the image of
$(\lambdabar_1, \lambdabar_1 - \lambdabar_2)$ on $E'$.

The action of the involution $w_6$ is the composition of that of $w_2$ and
$w_3$. It is given by
\begin{equation}
  w_6 (E, H_3, (\lambdabar_1, \lambdabar_2))
  = (E', H'_3, (\lambdabar'_1, \lambdabar'_2)) ,
\end{equation}
where $E'$ is the quotient of $E$ by the group generated by $H_3$ and $2
\lambdabar_1$, where $H'_3$ is the image of $E [3]$ on $E'$, and where
$(\lambdabar'_1, \lambdabar'_2)$ is some point on $E'$ such that
$(\lambdabar'_1, 2 \lambdabar'_2)$ equals the image of $(\lambdabar_2,
\lambdabar_1)$ on $E'$.

\subsection{Case III}

In this case the order $\O''$ is of index $32$, given by
\begin{equation}
  \O_{32} = \langle
  \smat{1}{0}{0}{1},
  \smat{0}{2}{-2}{0},
  \smat{2}{-1}{1}{-2},
  \smat{0}{1}{3}{0}
  \rangle
\end{equation}
We obtain a congruence subgroup of level $8$. There is a single involution
defined by the element
\begin{equation}
  w_2 = \smat{-1}{-3}{1}{1}
\end{equation}
of determinant $2$. In this case we get a modular description in terms of the
$8$-torsion for which a simpler description than the generic one in Section
\ref{sec:modgen} does not seem readily available.

\subsection{Case IV}

We consider the order $\O''$ obtained by enlarging $\Gamma^{(2)}$. This is the
intersection of the index $4$ order
\begin{equation}
  \O_4 = \langle
  \smat{1}{0}{0}{1},
  \smat{1}{0}{0}{-1},
  \smat{0}{1}{1}{1},
  \smat{0}{1}{-1}{1}
  \rangle
\end{equation}
and the index $9$ order
\begin{equation}
  \O_9 = \langle
  \smat{1}{0}{0}{1},
  \smat{0}{1}{0}{1},
  \smat{1}{1}{0}{-1},
  \smat{0}{0}{3}{0}
  \rangle .
\end{equation}
We obtain a congruence subgroup of level $6$, from which we can obtain $\Gamma$
by adding the commutators $\gamma_1 = \smat{1}{1}{2}{1}$ and $\gamma_2 =
\smat{1}{-1}{-2}{1}$. This group corresponds to the unique normal subgroup of
$\SL_2 (\Z / 6 \Z)$ that gives rise to a quotient that is cyclic of order $6$.
Using the corresponding character makes it easier to check the equivalence
\eqref{eq:equivgen}.

\bibliographystyle{plain}
\bibliography{canmod-1inf.bib}

\begin{thebibliography}{10}

\bibitem{magma}
Wieb Bosma, John Cannon, and Catherine Playoust.
\newblock The {M}agma algebra system. {I}. {T}he user language.
\newblock {\em J. Symbolic Comput.}, 24(3-4):235--265, 1997.
\newblock Computational algebra and number theory (London, 1993).

\bibitem{conrad-sl2}
Keith Conrad.
\newblock $\textrm{SL}_2 (\mathbf{Z})$.
\newblock Expository note available at
  \url{http://www.math.uconn.edu/~kconrad/blurbs}.

\bibitem{deligne}
Pierre Deligne.
\newblock Vari\'et\'es de {S}himura: interpr\'etation modulaire, et techniques
  de construction de mod\`eles canoniques.
\newblock In {\em Automorphic forms, representations and {$L$}-functions
  ({P}roc. {S}ympos. {P}ure {M}ath., {O}regon {S}tate {U}niv., {C}orvallis,
  {O}re., 1977), {P}art 2}, Proc. Sympos. Pure Math., XXXIII, pages 247--289.
  Amer. Math. Soc., Providence, R.I., 1979.

\bibitem{deligne-rapoport}
Pierre Deligne and Michael Rapoport.
\newblock Les sch\'emas de modules de courbes elliptiques.
\newblock pages 143--316. Lecture Notes in Math., Vol. 349, 1973.

\bibitem{fesenko}
Ivan Fesenko.
\newblock Arithmetic deformation theory via arithmetic fundamental groups and
  nonarchimedean theta functions, notes on the work of {S}hinichi {M}ochizuki.
\newblock {\em Europ. J. Math}, (1):405--440, 2015.

\bibitem{kmssv}
Michael Klug, Michael Musty, Sam Schiavone, Jeroen Sijsling, and John Voight.
\newblock Computing a database of {B}ely\u{\i} maps.
\newblock Unpublished preprint, 2017.

\bibitem{kmsv}
Michael Klug, Michael Musty, Sam Schiavone, and John Voight.
\newblock Numerical calculation of three-point branched covers of the
  projective line.
\newblock {\em LMS J. Comput. Math.}, 17(1):379--430, 2014.

\bibitem{lmfdb}
The {LMFDB Collaboration}.
\newblock The {L}-functions and {M}odular {F}orms {D}atabase.
\newblock \url{http://www.lmfdb.org}, 2017.
\newblock [Online; accessed 30 June 2017].

\bibitem{mochizuki-corresps}
Shinichi Mochizuki.
\newblock Correspondences on hyperbolic curves.
\newblock {\em J. Pure Appl. Algebra}, 131(3):227--244, 1998.

\bibitem{mochizuki-canonical}
Shinichi Mochizuki.
\newblock The absolute anabelian geometry of canonical curves.
\newblock {\em Doc. Math.}, (Extra Vol.):609--640, 2003.
\newblock In honor of Kazuya Kato's fiftieth birthday.

\bibitem{mochizuki-iut}
Shinichi Mochizuki.
\newblock Inter-universal {T}eichmüller theory {I}: {C}onstruction of {H}odge
  theaters.
\newblock Preprint available at
  \url{http://www.kurims.kyoto-u.ac.jp/~motizuki/papers-english.html}, 2017.

\bibitem{shimura}
Goro Shimura.
\newblock On canonical models of arithmetic quotients of bounded symmetric
  domains.
\newblock {\em Ann. of Math. (2)}, 91:144--222, 1970.

\bibitem{sijsling-thesis}
Jeroen Sijsling.
\newblock {\em Equations for arithmetic pointed tori}.
\newblock PhD thesis, Universiteit Utrecht, 2010.
\newblock Available at
  \url{https://sites.google.com/site/sijsling/research/thesis-sijsling.pdf}.

\bibitem{sijsling-belyi}
Jeroen Sijsling.
\newblock Arithmetic {$(1;e)$}-curves and {B}ely\u\i \ maps.
\newblock {\em Math. Comp.}, 81(279):1823--1855, 2012.

\bibitem{sijsling-1e}
Jeroen Sijsling.
\newblock Canonical models of arithmetic {$(1;e)$}-curves.
\newblock {\em Math. Z.}, 273(1-2):173--210, 2013.

\bibitem{sijsling-1inf-code}
Jeroen Sijsling.
\newblock Canonical models of arithmetic (1; inf)-curves.
\newblock Code available at \url{https://github.com/JRSijsling/canmod-1inf},
  2017.

\bibitem{sv-survey}
Jeroen Sijsling and John Voight.
\newblock On computing {B}elyi maps.
\newblock In {\em Num\'ero consacr\'e au trimestre ``{M}\'ethodes
  arithm\'etiques et applications'', automne 2013}, volume 2014/1 of {\em Publ.
  Math. Besan\c{c}on Alg\`ebre Th\'eorie Nr.}, pages 73--131. Presses Univ.
  Franche-Comt\'e, Besan\c{c}on, 2014.

\bibitem{takeuchi-char}
Kisao Takeuchi.
\newblock A characterization of arithmetic {F}uchsian groups.
\newblock {\em J. Math. Soc. Japan}, 27(4):600--612, 1975.

\bibitem{takeuchi-1e}
Kisao Takeuchi.
\newblock Arithmetic {F}uchsian groups with signature {$(1;e)$}.
\newblock {\em J. Math. Soc. Japan}, 35(3):381--407, 1983.

\bibitem{voight-matrix}
John Voight.
\newblock Identifying the matrix ring: algorithms for quaternion algebras and
  quadratic forms.
\newblock In {\em Quadratic and higher degree forms}, volume~31 of {\em Dev.
  Math.}, pages 255--298. Springer, New York, 2013.

\end{thebibliography}

\end{document}